\documentclass[11pt]{article}
\usepackage{xcolor} 
\usepackage{amsfonts}
\usepackage{amsmath}
\usepackage{amsthm}
\usepackage{caption}
\usepackage{chngcntr}
\counterwithin{figure}{section}
\usepackage{tikz}
\usepackage{graphicx} 
\usepackage{comment}

\newcommand{\RED} {\hbox{\color{red} RED}}
\newcommand{\BLUE} {\hbox{\color{blue} BLUE}}
\newcommand{\BLACK} {\hbox{\color{black} BLACK}}

\newcommand{\COL}{\text{COL}}
\newcommand{\N}{\sf N}

\theoremstyle{plain}
\newtheorem{theorem}{Theorem}[section]

\theoremstyle{definition}
\newtheorem{definition}[theorem]{Definition}

\theoremstyle{plain}
\newtheorem{notation}[theorem]{Notation}

\theoremstyle{plain}
\newtheorem{lemma}[theorem]{Lemma}

\theoremstyle{plain}
\newtheorem{proposition}[theorem]{Proposition}

\title{The Induced Bipartite Ramsey Theorem:\\ An Exposition} 
\author{William Gasarch and Gary Peng}
\date{}

\begin{document}

\maketitle

\section{Introduction} 

This paper presents a proof of the Induced Bipartite Ramsey Theorem. Although based on the treatment in Diestel~\cite{Diestel-2017}, our treatment differs from his in the following ways: 
\begin{enumerate}
    \item 
    We provide more detail.
    \item 
    We provide examples on small graphs. 
    \item 
    We include the Ramsey Theorem for general (i.e., not necessarily induced) bipartite graphs for contrast. 
    \item 
     We view the Induced Bipartite Ramsey Theorem as a theorem of interest in and of itself. (Diestel presents the Induced Bipartite Ramsey Theorem as a lemma for the Induced Ramsey Theorem for general graphs.)
    \item 
    We do not prove the Induced Ramsey Theorem for general graphs. 
\end{enumerate}

To summarize, our goal is pedagogical.\\

In both our treatment and Diestel's, we only deal with the case of 2 colors.
Extending the results to $c$ colors is a straightforward but tedious exercise. 

\section{Bipartite Graphs}

\subsection{Preliminaries}
\begin{notation}
    Let $n \in \N$. We denote the set $\{1, \ldots, n\}$ by $[n].$
\end{notation}
\begin{notation}
Let $X$ be a set and $k \in \N$. We denote the set of $k$-subsets of $X$ by $\displaystyle\binom{X}{k}$.
\end{notation}
\begin{definition}
    A {\it bipartite graph} is a graph $G=(V,E)$ where (1) $V$ can be partitioned into disjoint sets $A,B$ and (2)  $E\subseteq A\times B$. We can think of $A$ as being on the left side of $G$ and $B$ as being on the right side of $G$. Figure~\ref{fig:bip} shows a bipartite graph. 
    \begin{figure}[!h]
        \centering
        \begin{tikzpicture}
            \tikzstyle{vertex} = [circle, draw]

            \draw node (A) at (0, 0.75) {A};
            \draw node[vertex] (1) at (0, 0) {1};
            \draw node[vertex] (2) at (0, -1.25) {2};
            \draw node[vertex] (3) at (0, -2.5) {3};

            \draw node (B) at (2, 0.75) {B};
            \draw node[vertex] (4) at (2, 0) {4};
            \draw node[vertex] (5) at (2, -1.25) {5};
            \draw node[vertex] (6) at (2, -2.5) {6};

            \draw[-] (1) -- (4);
            \draw[-] (1) -- (5);
            \draw[-] (4) -- (2);
            \draw[-] (2) -- (5);
            \draw[-] (2) -- (6);
            \draw[-] (5) -- (3);
            \draw[-] (3) -- (6);
        \end{tikzpicture}
        \caption{A Bipartite Graph}
        \label{fig:bip}
    \end{figure}
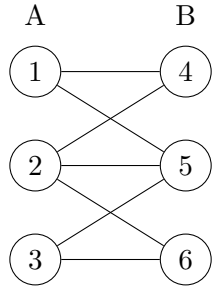
\end{definition}
\begin{definition}
    Let $G = (A, B, E)$ be a bipartite graph. We say that $G$ is \textit{complete} if $E = A \times B,$ and in this case, we denote $G$ by $K_{|A|, |B|}.$ Figure~\ref{fig:k33} shows $K_{3,3}$. 

    \begin{figure}[!h]
        \centering
        \begin{tikzpicture}
            \tikzstyle{vertex} = [circle, draw]

            \draw node[vertex] (1) at (0, 0) {1};
            \draw node[vertex] (2) at (0, -1.25) {2};
            \draw node[vertex] (3) at (0, -2.5) {3};

            \draw node[vertex] (4) at (2, 0) {4};
            \draw node[vertex] (5) at (2, -1.25) {5};
            \draw node[vertex] (6) at (2, -2.5) {6};

            \draw[-] (1) -- (4);
            \draw[-] (1) -- (5);
            \draw[-] (1) -- (6);
            \draw[-] (2) -- (4);
            \draw[-] (2) -- (5);
            \draw[-] (2) -- (6);
            \draw[-] (3) -- (4);
            \draw[-] (3) -- (5);
            \draw[-] (3) -- (6);
        \end{tikzpicture}
        \caption{$K_{3, 3}$}
        \label{fig:k33}
    \end{figure}
\end{definition}
\begin{notation}
    Let $n \in \N$ and $k \in \N$ with $k \le n.$ We denote the bipartite graph 
    $$G = \left([n], \binom{[n]}{k}, E\right),$$
    where 
    $$E = \left\{(x, X) \in [n] \times \binom{[n]}{k}  \colon x \in X\right\},$$ 
    by $B_{n, k}.$ Figure~\ref{fig:B42} shows $B_{4,2}$. 

    \begin{figure}[!h]
        \centering
        \begin{tikzpicture}
            \tikzstyle{vertex1} = [circle, draw]
            \tikzstyle{vertex2} = [circle, draw, scale=0.75]

            \draw node[vertex1] (1) at (0, 0) {1};
            \draw node[vertex1] (2) at (0, -1) {2};
            \draw node[vertex1] (3) at (0, -2) {3};
            \draw node[vertex1] (4) at (0, -3) {4};

            \draw node[vertex2] (5) at (2, 1) {1, 2};
            \draw node[vertex2] (6) at (2, 0) {1, 3};
            \draw node[vertex2] (7) at (2, -1) {1, 4};
            \draw node[vertex2] (8) at (2, -2) {2, 3};
            \draw node[vertex2] (9) at (2, -3) {2, 4};
            \draw node[vertex2] (10) at (2, -4) {3, 4};

            \draw[-] (1) -- (5);
            \draw[-] (1) -- (6);
            \draw[-] (1) -- (7);
            \draw[-] (2) -- (5);
            \draw[-] (2) -- (8);
            \draw[-] (2) -- (9);
            \draw[-] (3) -- (6);
            \draw[-] (3) -- (8);
            \draw[-] (3) -- (10);
            \draw[-] (4) -- (7);
            \draw[-] (4) -- (9);
            \draw[-] (4) -- (10);
        \end{tikzpicture}
        \caption{$B_{4, 2}$}
        \label{fig:B42}
    \end{figure}
\end{notation}

We now look at two different ways a graph can be a subgraph of another graph.

\begin{notation}
    Let $G = (V, E)$ be a graph and $H=(V',E')$ be a graph with $V'\subseteq V.$
    \begin{enumerate}
        \item 
        We say that $H$ is a {\it subgraph} of $G$ if $\displaystyle E' \subseteq \binom{V'}{2} \cap E$. 
        \item 
        We say that $H$ is an {\it induced subgraph} of $G$ if $\displaystyle E'= \binom{V'}{2} \cap E$.
        \end{enumerate} 
        Figures~\ref{fig:notinduced1} and \ref{fig:notinduced2} show subgraphs that are not induced subgraphs. Figures~\ref{fig:induced1} and \ref{fig:induced2} show subgraphs that are induced subgraphs. 
\end{notation}
\begin{figure}[!h]
    \centering
    \begin{tikzpicture}
        \tikzstyle{vertex} = [circle, draw]

        \draw node[vertex, red] (1) at (0, 0) {1};
        \draw node[vertex, red] (2) at (2, 0) {2};
        \draw node[vertex] (3) at (0, -2) {3};
        \draw node[vertex, red] (4) at (2, -2) {4};

        \draw[-, red] (1) -- (2);
        \draw[-] (1) -- (3);
        \draw[-, red] (1) -- (4);
        \draw[-] (2) -- (4);
        \draw[-] (3) -- (4);
    \end{tikzpicture}
    \caption{The red subgraph is not an induced subgraph.}
    \label{fig:notinduced1}
\end{figure}
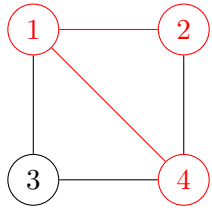
\begin{figure}[!h]
    \centering
    \begin{tikzpicture}
        \tikzstyle{vertex} = [circle, draw]

        \draw node[vertex, blue] (1) at (0, 0) {1};
        \draw node[vertex, blue] (2) at (2, 0) {2};
        \draw node[vertex] (3) at (0, -2) {3};
        \draw node[vertex, blue] (4) at (2, -2) {4};

        \draw[-, blue] (1) -- (2);
        \draw[-] (1) -- (3);
        \draw[-, blue] (1) -- (4);
        \draw[-, blue] (2) -- (4);
        \draw[-] (3) -- (4);
    \end{tikzpicture}
    \caption{The blue subgraph is an induced subgraph.}
    \label{fig:induced1} 
\end{figure}
\begin{figure}[!h]
    \centering
    \begin{tikzpicture}
        \tikzstyle{vertex} = [circle, draw]

        \draw node[vertex, red] (1) at (0, 0) {1};
        \draw node[vertex, red] (2) at (0, -2) {2};
        \draw node[vertex, red] (3) at (0, -4) {3};
        \draw node[vertex, red] (4) at (2, 0) {4};
        \draw node[vertex, red] (5) at (2, -2) {5};
        \draw node[vertex] (6) at (2, -4) {6};
        
        \draw[-, red] (1) -- (4);
        \draw[-] (1) -- (5);
        \draw[-] (1) -- (6);
        \draw[-, red] (2) -- (5);
        \draw[-] (3) -- (4);
        \draw[-, red] (3) -- (5);
        \draw[-] (3) -- (6);
    \end{tikzpicture}
    \caption{The red subgraph is not an induced subgraph.}
    \label{fig:notinduced2}
\end{figure}
\begin{figure}[!h]
    \centering
    \begin{tikzpicture}
        \tikzstyle{vertex} = [circle, draw]

        \draw node[vertex, blue] (1) at (0, 0) {1};
        \draw node[vertex, blue] (2) at (0, -2) {2};
        \draw node[vertex, blue] (3) at (0, -4) {3};
        \draw node[vertex, blue] (4) at (2, 0) {4};
        \draw node[vertex, blue] (5) at (2, -2) {5};
        \draw node[vertex] (6) at (2, -4) {6};
        
        \draw[-, blue] (1) -- (4);
        \draw[-, blue] (1) -- (5);
        \draw[-] (1) -- (6);
        \draw[-, blue] (2) -- (5);
        \draw[-, blue] (3) -- (4);
        \draw[-, blue] (3) -- (5);
        \draw[-] (3) -- (6);
    \end{tikzpicture}
    \caption{The blue subgraph is an induced subgraph.}
    \label{fig:induced2}
\end{figure}
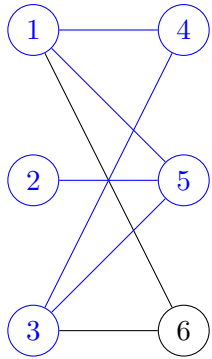

\section{Hypergraphs}

We will need the following concepts:

\begin{definition}
Let $a\ge 1$. An {\it $a$-ary hypergraph} is a pair $(V,E)$ where $V$ is a set and $E\subseteq \binom{V}{a}$. Note that a graph is a 2-ary hypergraph.
\end{definition}

\begin{definition}
    Let $G = (V, E)$ be an $a$-ary hypergraph and $\COL: E \rightarrow [c]$ be a $c$-coloring of $E.$ We say that $H \subseteq V$ is a \textit{homogeneous set} if $\COL$ restricted to $\displaystyle \binom{H}{a} \cap E$ is constant.
\end{definition}

We now state the $a$-ary Ramsey Theorem:

\begin{theorem}\label{th:aary} 
Let $a,c,k\in\N$. Then there exists $n$ such that, 
for all $c$-colorings of $\binom{[n]}{a}$, there exists a homogenous set of size $k$.
\end{theorem}

\begin{notation}\label{no:Rack}
    Let $a, c, k \in \N.$ We denote by $R_{a, c}(k)$ the minimum $n$ that is guaranteed to exist by Theorem~\ref{th:aary}. 
\end{notation}

\subsection{The Bipartite Ramsey Theorem} 

The following theorem is folklore:

\begin{theorem}\label{th:BipRamThm} 
   For all $a,b \in \N$, there exists $n, k \in \N$ such that for all $2$-colorings of the edges of $K_{n, k},$ there exists a monochromatic $K_{a, b}.$ In particular, we can take $k=2b$ and 
   $n=a2^{2b}$.
\end{theorem}

\begin{proof}  
    Let $\COL: [n] \times [k] \rightarrow [2]$ be a $2$-coloring of the edges of 
    $K_{n, k},$ where $n, k \in \N$ are constants to be determined later, and define $\COL': [n] \rightarrow [2^{k}]$ as $$\COL'(x) = (\COL(x, 1), \ldots, \COL(x, k)).$$ Then, since there are $n$ elements in the domain and $2^{k}$ elements in the co-domain of $\COL'$, some color $\vec c = (c_1, ..., c_{k})$ is mapped to by at least $\frac{n}{2^{k}}$ left-vertices in $K_{n, k}$. (We will later pick $n,k$ such that $2^{k}$ divides $n$.) Thus, since $c_i \in \{\RED,\BLUE\}$, there exists $c^*\in \{\RED,\BLUE\}$ that appears at least $\frac{k}{2}$ times in $\vec c.$ (We will later pick $k$ such that 2 divides $k$.) Hence, $\vec c$ looks something like $$\vec c = (\cdots c^* \cdots c^* \cdots c^* \cdots \cdots\cdots c^* \cdots).$$
    Now, suppose that the $c^*$ occur at positions $p_1,\ldots,p_{\frac{k}{2}}$. Then, the $\frac{n}{2^{k}}$ left-vertices that map to $\vec c$ and the right-vertices $p_1,\ldots,p_{\frac{k}{2}}$ induce a monochromatic complete bipartite graph.\\
    
To finish the proof, we will now pick $n,k$. We need the following:
\begin{enumerate}
    \item 
    $\frac{n}{2^{k}}\ge a$ and $2^k | n$
    \item 
    $\frac{k}{2}\ge b$ and $2 | k$
\end{enumerate}
Therefore, we can take $k=2b$ and $n=a2^{k}=a2^{2b}$.
   \end{proof}

\section{The Induced $B_{4,2}$ Ramsey Theorem}

By Theorem~\ref{th:BipRamThm}, for all $a \in \N$, there exists 
$n,k \in \N$ such that for all 2-colorings of the edges of $K_{n,k}$, there exists a monochromatic $K_{a,a}$. Note that this is also an induced monochromatic $K_{a,a}$.

What if we wanted a monochromatic $B_{4,2}$? Since $B_{4,2}$ is a subgraph of $K_{4,6}$, we can just force a monochromatic $K_{4,6}$ and take the monochromatic $B_{4,2}$ inside it. 

What if we wanted the monochromatic $B_{4,2}$ to be \textit{induced}? (See Figure~\ref{fig:want} for the difference between a monochromatic $K_{4,6}$ and a monochromatic $B_{4,2}$). In this case, \textit{no complete bipartite graph  works} since $B_{4,2}$ cannot be an induced subgraph of any complete bipartite graph. So, what can we do? As it turns out, we can use a bipartite graph that is not complete. In fact, we will use a $B_{n,k}$. 

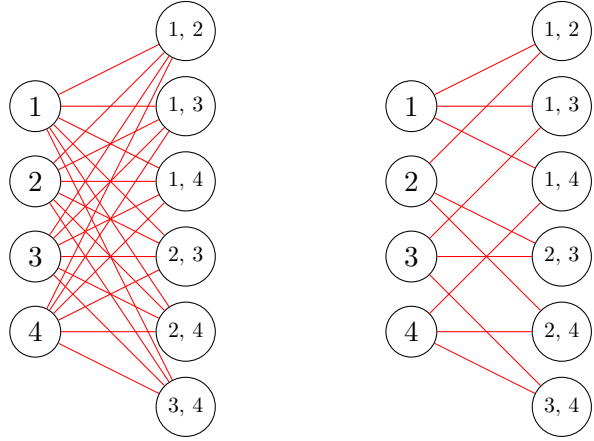
\begin{figure}[!h]
        \centering
        \begin{tikzpicture}
            \tikzstyle{vertex1} = [circle, draw]
            \tikzstyle{vertex2} = [circle, draw, scale=0.75]

            \draw node[vertex1] (21) at (-3.5, -1) {1};
            \draw node[vertex1] (22) at (-3.5, -2) {2};
            \draw node[vertex1] (23) at (-3.5, -3) {3};
            \draw node[vertex1] (24) at (-3.5, -4) {4};

            \draw node[vertex2] (25) at (-1.5, 0) {1, 2};
            \draw node[vertex2] (26) at (-1.5, -1) {1, 3};
            \draw node[vertex2] (27) at (-1.5, -2) {1, 4};
            \draw node[vertex2] (28) at (-1.5, -3) {2, 3};
            \draw node[vertex2] (29) at (-1.5, -4) {2, 4};
            \draw node[vertex2] (30) at (-1.5, -5) {3, 4};

            \draw[-, red] (21) -- (25);
            \draw[-, red] (21) -- (26);
            \draw[-, red] (21) -- (27);
            \draw[-, red] (21) -- (28);
            \draw[-, red] (21) -- (29);
            \draw[-, red] (21) -- (30);
            \draw[-, red] (22) -- (25);
            \draw[-, red] (22) -- (26);
            \draw[-, red] (22) -- (27);
            \draw[-, red] (22) -- (28);
            \draw[-, red] (22) -- (29);
            \draw[-, red] (22) -- (30);
            \draw[-, red] (23) -- (25);
            \draw[-, red] (23) -- (26);
            \draw[-, red] (23) -- (27);
            \draw[-, red] (23) -- (28);
            \draw[-, red] (23) -- (29);
            \draw[-, red] (23) -- (30);
            \draw[-, red] (24) -- (25);
            \draw[-, red] (24) -- (26);
            \draw[-, red] (24) -- (27);
            \draw[-, red] (24) -- (28);
            \draw[-, red] (24) -- (29);
            \draw[-, red] (24) -- (30);

            \draw node[vertex1] (1) at (1.5, -1) {1};
            \draw node[vertex1] (2) at (1.5, -2) {2};
            \draw node[vertex1] (3) at (1.5, -3) {3};
            \draw node[vertex1] (4) at (1.5, -4) {4};

            \draw node[vertex2] (5) at (3.5, 0) {1, 2};
            \draw node[vertex2] (6) at (3.5, -1) {1, 3};
            \draw node[vertex2] (7) at (3.5, -2) {1, 4};
            \draw node[vertex2] (8) at (3.5, -3) {2, 3};
            \draw node[vertex2] (9) at (3.5, -4) {2, 4};
            \draw node[vertex2] (10) at (3.5, -5) {3, 4};

            \draw[-, red] (1) -- (5);
            \draw[-, red] (1) -- (6);
            \draw[-, red] (1) -- (7);
            \draw[-, red] (2) -- (5);
            \draw[-, red] (2) -- (8);
            \draw[-, red] (2) -- (9);
            \draw[-, red] (3) -- (6);
            \draw[-, red] (3) -- (8);
            \draw[-, red] (3) -- (10);
            \draw[-, red] (4) -- (7);
            \draw[-, red] (4) -- (9);
            \draw[-, red] (4) -- (10);
        \end{tikzpicture}
        \caption{a monochromatic $K_{4,6}$ (left) and a monochromatic $B_{4,2}$ (right)}
        \label{fig:want}
    \end{figure}

\begin{theorem}\label{th:IndB42} 
There exists $n,k\in\N$ such that for all 2-colorings of the edges of 
$B_{n,k}$, there exists an induced monochromatic $B_{4, 2}.$
\end{theorem}

\begin{proof}

Let 
$$n = R_{3, 2\binom{3}{2}}(9)=R_{3,6}(9).
 \footnote{Recall that $R_{3, 2\binom{3}{2}}(9)$ is the smallest $n \in \N$ such that for all $2\binom{3}{2}$-colorings of the $3$-ary hypergraph $\left([n], \binom{[n]}{3}\right),$ there exists a homogeneous set of size 9.}$$ We will show that any 2-coloring of the edges of $\displaystyle B_{n,3} =\left([n],\binom{[n]}{3},E\right),$ where $$E= \left\{ (x,X) \in [n] \times \binom{[n]}{3} \colon x\in X \right\}$$ has an induced monochromatic $B_{4,2}$ (i.e., $k=3$). Figure~\ref{fig:Bn3} shows $B_{n,3}$.
\begin{figure}[!h]
    \centering
    \begin{tikzpicture}
        \tikzstyle{vertex1} = [circle, draw, minimum width={width("$n - 1$") + 10pt}]
        \tikzstyle{vertex2} = [circle, draw, minimum width={width("$n - 2, n - 1, n$") + 10pt}]
        \tikzstyle{vertex3} = []

        \draw node[vertex1] (1) at (0, 0) {1};
        \draw node[vertex1] (2) at (0, -2) {2};
        \draw node[vertex1] (3) at (0, -4) {3};
        \draw node[vertex1] (4) at (0, -6) {4};
        \draw node[vertex1] (5) at (0, -8.5) {$n - 3$};
        \draw node[vertex1] (6) at (0, -10.5) {$n - 2$};
        \draw node[vertex1] (7) at (0, -12.5) {$n - 1$};
        \draw node[vertex1] (8) at (0, -14.5) {$n$};

        \draw node[vertex2] (11) at (4, -1) {$1, 2, 3$};
        \draw node[vertex2] (12) at (4, -5) {$1, 2, 4$};
        \draw node[vertex2] (13) at (4, -9.5) {$n - 3, n - 2, n$};
        \draw node[vertex2] (14) at (4, -13.5) {$n - 2, n - 1, n$};

        \path (4) -- (5) node [midway, sloped] {$\dots$};
        \path (12) -- (13) node [midway, sloped] {$\dots$};
        
        \draw (1) -- (11);
        \draw (2) -- (11);
        \draw (3) -- (11);
        \draw (1) -- (12);
        \draw (2) -- (12);
        \draw (4) -- (12);
        
        \draw (5) -- (13);
        \draw (6) -- (13);
        \draw (8) -- (13);
        \draw (6) -- (14);
        \draw (7) -- (14);
        \draw (8) -- (14);
        
    \end{tikzpicture}
    \caption{$B_{n, 3}$}
    \label{fig:Bn3}
\end{figure}

Let $\displaystyle \COL: E \rightarrow [2]$ be a $2$-coloring of the edges of $B_{n, 3}$. We will define $\displaystyle \COL'\colon \binom{[n]}{3}\rightarrow \left[2\binom{3}{2}\right]$ for $\displaystyle X = \{z_1 < z_2 < z_3\} \in \binom{[n]}{3}$. In particular, since there are edges of the form $(z_1,X)$, $(z_2,X)$ and $(z_3,X)$ in $B_{n,3}$, $2$ of these edges -- say $(z_i, X)$ and $(z_j, X)$ -- have the same color $c$ under $\COL$. We define $\COL'(X)=(c,\{i,j\})$. 

Now, by the definition of $n$, there is a homogeneous set $H$ (relative to $\COL'$) of size 9 and color $(c,\{i, j\}).$ Without loss of generality, we may assume that $H=\{1,\ldots,9\}$ and that $c = \RED.$ We will now work out the case where $(i,j)=(1,3)$ and sketch the other cases later. 

\bigskip

\noindent
{\bf Case 1: $(i,j)=(1,3)$}\\\\
\noindent
In this case, for all right-vertices $X \subseteq H$ in $B_{n, 3},$ the 1st and 3rd edges coming into $X$ are $\RED$. However, we won't assume anything about the color of the 2nd edge coming into $X$. Thus, in the figures below, we will, as the Rolling Stones would do, paint the edges whose colors we won't assume anything about $\BLACK$. 

Figure~\ref{fig:B93} shows $B_{n, 3}$ restricted to the left-vertices $1,\ldots,4.$
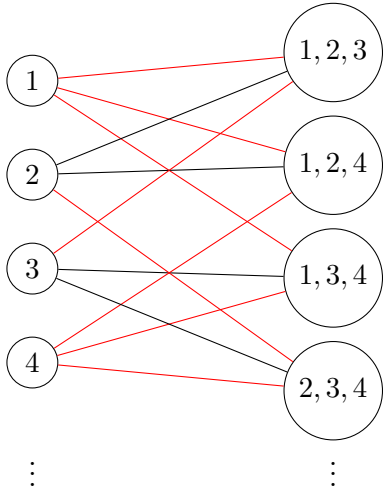
\begin{figure}[!h]
    \centering
    \begin{tikzpicture}
        \tikzstyle{vertex} = [circle, draw]

        \draw node[vertex] (1) at (0, -0.375) {1};
        \draw node[vertex] (2) at (0, -1.625) {2};
        \draw node[vertex] (3) at (0, -2.875) {3};
        \draw node[vertex] (4) at (0, -4.125) {4};

        \draw node[vertex] (11) at (4, 0) {$1, 2, 3$};
        \draw node[vertex] (12) at (4, -1.5) {$1, 2, 4$};
        \draw node[vertex] (13) at (4, -3) {$1, 3, 4$};
         \draw node[vertex] (14) at (4, -4.5) {$2, 3, 4$};

        \draw node (5) at (0, -5.5) {};
        \draw node (15) at (4, -5.5) {};

        \path (4) -- (5) node [] {$\vdots$};
        \path (14) -- (15) node [] {$\vdots$};
        
        \draw[red, -] (1) -- (11);
        \draw[red, -] (1) -- (12);
        \draw[red, -] (1) -- (13);
        \draw[-] (2) -- (11);
        \draw[-] (2) -- (12);
        \draw[red, -] (3) -- (11);
        \draw[-] (3) -- (13);
        \draw[red, -] (4) -- (12);
        \draw[red, -] (4) -- (13);
        \draw[red, -] (2) -- (14);
        \draw[-] (3) -- (14);
        \draw[red, -] (4) -- (14);
    \end{tikzpicture}
    \caption{$B_{n, 3}$ restricted to the left-vertices $1, \ldots, 4$}
    \label{fig:B93}
\end{figure}

\newpage

Now, consider the subgraph induced by left-vertices $2,4,6,8$ and the right-vertices $\{2,3,4\}$, $\{2,3,6\}$, $\{2,3,8\}$, $\{4,5,6\}$, $\{4,5,8\}$, $\{6,7,8\}$ in $B_{n, 3}$. Note that this subgraph is a monochromatic $\RED$ $B_{4, 2},$ as shown in Figure~\ref{fig:indB42-13}. Indeed, for any right-vertex $\{a<b<c\}$  in this subgraph, the middle entry $b$ is absent from the left-vertex set. We now see why we wanted a homogeneous set of size 9: We wanted a set of four left-vertices such that (1) any pair of adjacent vertices has ONE missing vertex between them (2) the top vertex has ONE missing vertex above it and (3) the bottom vertex has ONE missing vertex below it. This ONE will be what changes when we generalize. 

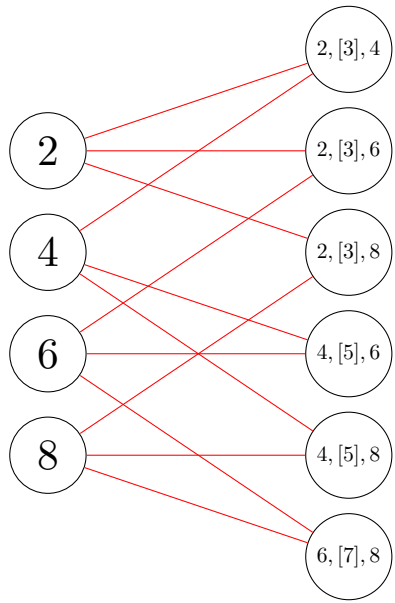
\begin{figure}[!h]
    \centering
    \begin{tikzpicture}
        \tikzstyle{vertex1} = [circle, draw, scale=1.5]
        \tikzstyle{vertex2} = [circle, draw, scale=0.75]

        \draw node[vertex1] (2) at (0, 0) {2};
        \draw node[vertex1] (4) at (0, -1.35) {4};
        \draw node[vertex1] (6) at (0, -2.7) {6};
        \draw node[vertex1] (8) at (0, -4.05) {8};
        \draw node[vertex2] (11) at (4, 1.35) {$2, [3], 4$};
        \draw node[vertex2] (12) at (4, 0) {$2, [3], 6$};
        \draw node[vertex2] (13) at (4, -1.35) {$2, [3], 8$};
        \draw node[vertex2] (14) at (4, -2.7) {$4, [5], 6$};
        \draw node[vertex2] (15) at (4, -4.05) {$4, [5], 8$};
        \draw node[vertex2] (16) at (4, -5.4) {$6, [7], 8$};

        \draw[red, -] (2) -- (11);
        \draw[red, -] (2) -- (12);
        \draw[red, -] (2) -- (13);
        \draw[red, -] (4) -- (11);
        \draw[red, -] (4) -- (14);
        \draw[red, -] (4) -- (15);
        \draw[red, -] (6) -- (12);
        \draw[red, -] (6) -- (14);
        \draw[red, -] (6) -- (16);
        \draw[red, -] (8) -- (13);
        \draw[red, -] (8) -- (15);
        \draw[red, -] (8) -- (16);
    \end{tikzpicture}
    \caption{the induced $B_{4, 2}$ in the $(i,j)=(1,3)$ case}
    \label{fig:indB42-13}
\end{figure}
\ \\
\noindent
{\bf Case 2: $(i,j)=(1,2)$}\\\\ 
\noindent
See Figure~\ref{fig:indB42-12}.\\\\
\noindent
{\bf Case 3: $(i,j)=(2,3)$}\\\\
See Figure~\ref{fig:indB42-23}.
\end{proof}
\begin{figure}[!h]
    \centering
    \begin{tikzpicture}
        \tikzstyle{vertex1} = [circle, draw, scale=1.5]
        \tikzstyle{vertex2} = [circle, draw, scale=0.75]

        \draw node[vertex1] (2) at (0, 0) {2};
        \draw node[vertex1] (4) at (0, -1.35) {4};
        \draw node[vertex1] (6) at (0, -2.7) {6};
        \draw node[vertex1] (8) at (0, -4.05) {8};
        \draw node[vertex2] (11) at (4, 1.35) {$2, 4, [5]$};
        \draw node[vertex2] (12) at (4, 0) {$2, 6, [7]$};
        \draw node[vertex2] (13) at (4, -1.35) {$2, 8, [9]$};
        \draw node[vertex2] (14) at (4, -2.7) {$4, 6, [7]$};
        \draw node[vertex2] (15) at (4, -4.05) {$4, 8, [9]$};
        \draw node[vertex2] (16) at (4, -5.4) {$6, 8, [9]$};

        \draw[red, -] (2) -- (11);
        \draw[red, -] (2) -- (12);
        \draw[red, -] (2) -- (13);
        \draw[red, -] (4) -- (11);
        \draw[red, -] (4) -- (14);
        \draw[red, -] (4) -- (15);
        \draw[red, -] (6) -- (12);
        \draw[red, -] (6) -- (14);
        \draw[red, -] (6) -- (16);
        \draw[red, -] (8) -- (13);
        \draw[red, -] (8) -- (15);
        \draw[red, -] (8) -- (16);
    \end{tikzpicture}
    \caption{the induced $B_{4, 2}$ in the $(i,j)=(1,2)$ case}
    \label{fig:indB42-12}
\end{figure}
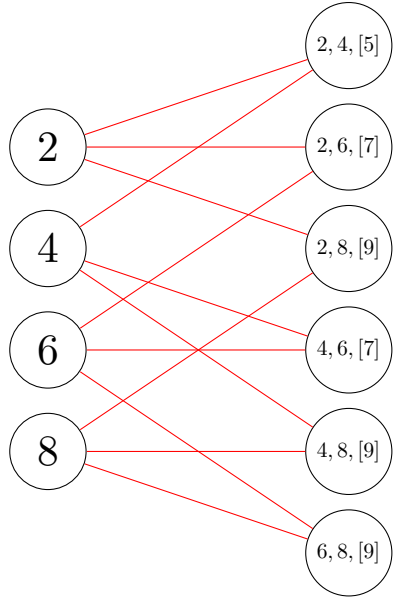

\begin{figure}[!h]
    \centering
    \begin{tikzpicture}
        \tikzstyle{vertex1} = [circle, draw, scale=1.5]
        \tikzstyle{vertex2} = [circle, draw, scale=0.75]

        \draw node[vertex1] (2) at (0, 0) {2};
        \draw node[vertex1] (4) at (0, -1.35) {4};
        \draw node[vertex1] (6) at (0, -2.7) {6};
        \draw node[vertex1] (8) at (0, -4.05) {8};
        \draw node[vertex2] (11) at (4, 1.35) {$[1],2, 4$};
        \draw node[vertex2] (12) at (4, 0) {$[1], 2, 6$};
        \draw node[vertex2] (13) at (4, -1.35) {$[1], 2, 8$};
        \draw node[vertex2] (14) at (4, -2.7) {$[3], 4, 6$};
        \draw node[vertex2] (15) at (4, -4.05) {$[3], 4, 8$};
        \draw node[vertex2] (16) at (4, -5.4) {$[5], 6, 8$};

        \draw[red, -] (2) -- (11);
        \draw[red, -] (2) -- (12);
        \draw[red, -] (2) -- (13);
        \draw[red, -] (4) -- (11);
        \draw[red, -] (4) -- (14);
        \draw[red, -] (4) -- (15);
        \draw[red, -] (6) -- (12);
        \draw[red, -] (6) -- (14);
        \draw[red, -] (6) -- (16);
        \draw[red, -] (8) -- (13);
        \draw[red, -] (8) -- (15);
        \draw[red, -] (8) -- (16);
    \end{tikzpicture}
    \caption{the induced $B_{4, 2}$ in the $(i,j)=(2,3)$ case}
    \label{fig:indB42-23}
\end{figure}
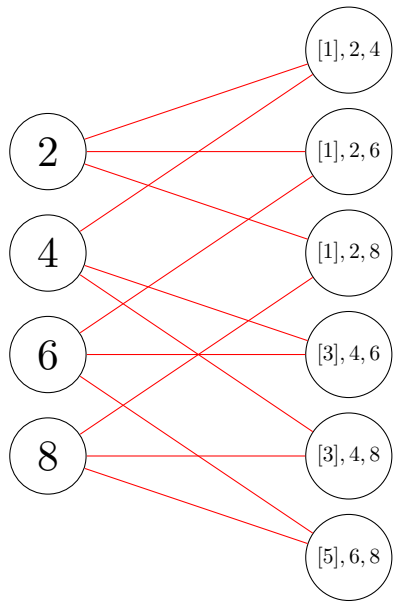

\section{The Induced $B_{n,k}$ Ramsey Theorem} 

We now generalize Theorem~\ref{th:IndB42} to apply to $B_{n,k}$ for any $n, k \in \N$. Most of the ideas used were already in the proof of Theorem~\ref{th:IndB42}.

\begin{theorem}\label{th:IndBab} 
   Let $a, b \in \N$ with $b \le a.$ 
   Then, there exists $n,k\in\N$ such that for all 2-colorings of the edges of  $B_{n,k}$, there exists an induced monochromatic $B_{a,b}.$
\end{theorem}

\begin{proof}{}
Let $s \in \N$ be a constant to be determined later, and let
$$n = R_{2b-1, 2\binom{2b-1}{b}}(s).\footnote{Recall that $R_{2b-1, 2\binom{2b-1}{b}}(s)$ is the smallest $n \in \N$ such that for all $2\binom{2b-1}{b}$-colorings of the $(2b-1)$-ary hypergraph $\left([n], \binom{[n]}{2b-1}\right)$, there exists a homogeneous set of size $s$.}$$ We will show that any 2-coloring of the edges of $\displaystyle B_{n,2b-1}=\left([n],\binom{[n]}{2b-1},E\right),$ where $$E= \left\{ (x,X) \in [n] \times \binom{[n]}{2b-1} \colon x\in X \right\}$$ has an induced monochromatic $B_{a,b}$ (i.e., $k=2b-1$). 

Let $\displaystyle \COL: E \rightarrow [2]$ be a $2$-coloring of the edges of $B_{n, 2b-1}$. We will define $\displaystyle\COL'\colon \binom{[n]}{2b-1}\rightarrow \left[2\binom{2b-1}{b}\right]$ for $\displaystyle X=\{z_1 < \cdots < z_{2b-1} \} \in \binom{[n]}{2b-1}$. In particular, since there are edges of the form $(z_1,X)$, $\cdots$, $(z_{2b-1},X)$ in $B_{n,2b-1}$, $b$ of these edges -- say $(z_{i_1}, X), \cdots, (z_{i_b}, X)$ -- have the same color $c$. We define
$\COL'(X)=(c,\{i_1,\ldots,i_b\})$.

Now, by the definition of $n,$ there exists a homogeneous set $H$ (relative to $\COL'$) of size $s$ and color $(c,\{i_1,\ldots,i_b\})$. Without loss of generality, we may assume that $H=\{1,\ldots,s\}$ and that $c = \RED$. Thus, for all right-vertices $X \subseteq H$ in $B_{n, 2b - 1}$, the $i_1$th, $\ldots$, $i_b$th edges coming into $X$ are $\RED$.

We will now choose $s$. Recall that in the proof of 
Theorem~\ref{th:IndB42}, we wanted an induced $B_{4,2}$ and chose $s=9$.
We then took the 4 (4 was the first number in $B_{4,2}$) left-vertices $2,4,6,8$ because there was ONE (ONE was one less than 2, the second number in $B_{4,2}$) invisible vertex between any pair of adjacent vertices, ONE invisible vertex above the top vertex, and ONE invisible vertex below the bottom vertex. Hence, we choose $s=ab+b-1$ so that we can take the left-vertices $b, 2b, \ldots, ab$ and have $b-1$ vertices after $ab$. We now leave the rest of the proof to the reader. 
\end{proof}

\
\section{The Induced Bipartite Ramsey Theorem For a Particular Graph} 

By Theorem~\ref{th:BipRamThm}, there is an Induced Ramsey Theorem for every bipartite graph {\it of the form $K_{a,b}$}. 
By Theorem~\ref{th:IndBab}, there is an Induced Ramsey Theorem for every bipartite graph {\it of the form $B_{a,b}$}. What about other bipartite graphs? In this section, we will consider the case of a small bipartite graph, the one in Figure~\ref{fig:simpleG}. Since it is just an example, we will call it a proposition instead of a theorem.

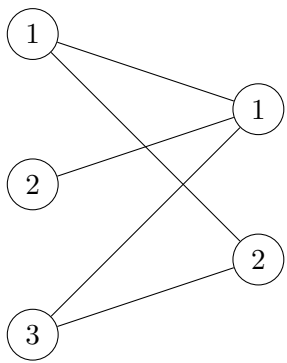
\begin{figure}[!h]
    \centering
    \begin{tikzpicture}
        \tikzstyle{vertex1} = [circle, draw]

        \draw node[vertex1] (1) at (0, 0) {$1$};
        \draw node[vertex1] (2) at (0, -2) {$2$};
        \draw node[vertex1] (3) at (0, -4) {$3$};

        \draw node[vertex1] (11) at (3, -1) {$1$};
        \draw node[vertex1] (12) at (3, -3) {$2$};
        
        \draw[-] (1) -- (11);
        \draw[-] (1) -- (12);
        \draw[-] (2) -- (11);
        \draw[-] (3) -- (11);
        \draw[-] (3) -- (12);
    \end{tikzpicture}
    \caption{A Small Bipartite Graph}
    \label{fig:simpleG}
\end{figure}

\begin{proposition}\label{pr:simpleG} 
Let $G$ be the bipartite graph in Figure~\ref{fig:simpleG}.
Then, there exists $n,k\in\N$ such that for all 2-colorings of the edges of 
$B_{n,k}$, there exists an induced monochromatic $G$. 
\end{proposition}

\begin{proof}
Let $n,k \in \N$ be such that for all 2-colorings of the edges of $B_{n,k}$, there exists an induced monochromatic $B_{7,4}$. (Note that such $n,k$ exist by Theorem~\ref{th:IndBab}.) Then, let $\displaystyle \COL: E \rightarrow [2]$ be a $2$-coloring of the edges of $B_{n, k}$, and label the left-vertices of the induced monochromatic $B_{7,4}$ as $1,2,3,1',2',3',1'',2''$. Now, consider the induced subgraph of this $B_{7,4}$ that has 
\begin{enumerate}
    \item 
    left-vertices $1,2,3$ 
    \item 
    right-vertices $\{1,2,3,1'\}$ and $\{1,3,2',1''\}$
\end{enumerate}
Figure~\ref{fig:simpleindG} shows this induced bipartite graph. Note that it is $G$. 
\end{proof}

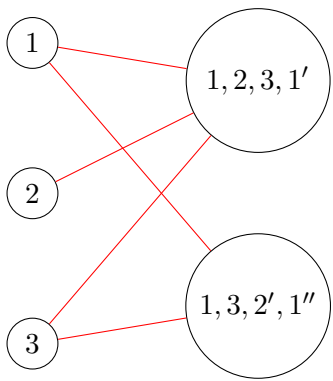
\begin{figure}[!h]
    \centering
    \begin{tikzpicture}
        \tikzstyle{vertex1} = [circle, draw]
        \tikzstyle{vertex2} = [circle, draw, minimum width={width("$1, 3, 2', 1''$") + 10pt}]

        \draw node[vertex1] (1) at (0, 0) {$1$};
        \draw node[vertex1] (2) at (0, -2) {$2$};
        \draw node[vertex1] (3) at (0, -4) {$3$};

        \draw node[vertex2] (11) at (3, -0.5) {$1, 2, 3, 1'$};
        \draw node[vertex2] (12) at (3, -3.5) {$1, 3, 2', 1''$};
        
        \draw[-, red] (1) -- (11);
        \draw[-, red] (1) -- (12);
        \draw[-, red] (2) -- (11);
        \draw[-, red] (3) -- (11);
        \draw[-, red] (3) -- (12);
    \end{tikzpicture}
    \caption{The Induced Bipartite Graph $G$}
    \label{fig:simpleindG}
\end{figure}

In this next section, we generalize Proposition~\ref{pr:simpleG} to any bipartite graph. The key to the proof of Proposition~\ref{pr:simpleG} is
that $B_{7,4}$ has an induced $G$, where $G$ is the 
bipartite graph in Figure~\ref{fig:simpleG}. We will now show that for any 
bipartite graph $G$, there exists $a,b \in \N$ such that $G$ is an induced subgraph of $B_{a,b}$. We will then prove the Induced Bipartite Ramsey Theorem, which will follow directly from this result.

\begin{lemma}\label{le:BipIndSubgraph}
Let $G=(C,D,E)$ be a bipartite graph. Then, there exists $a,b \in \N$ such that $G$ is an induced subgraph of $B_{a,b}$. 
\end{lemma} 

\begin{proof} 
Let $C= \{1,\ldots,c\}$ and $D=\{1,\ldots, d\}$. 
 Now, consider the graph $B_{a, b},$ where $a=2c+d$ and $b=c+1,$ and label the left-vertices as $\\1,\ldots,c, 1',\ldots,c', 1'',\ldots,d''.$ We will now show that $G$ is an induced subgraph of $B_{a,b}$ by giving the vertices that induce $G$:
\begin{enumerate}
    \item 
    The left-vertices are $1,\ldots,c$. These will correspond to 
    the left-vertices in $G,$ which are also named $1,\ldots,c$. 
    \item 
    For the right-vertices, we will form a $b$-set to correspond to each right-vertex $1\le j \le d$ in $G.$ We start by adding $\{z_1,\ldots,z_L\},$ the neighbors of $j$ in $G,$ to the set. (Note that $L \le c < c + 1 = b$.) Next, we add $j''$ to the set to distinguish it from the other right-vertices in $B_{a, b}$ that contain $\{z_1,\ldots,z_L\}.$ Note the set has $L+1$ elements now. To finish, we add $1',\ldots,(b-L-1)'$ to the set to bring the size up to $b$. 
\end{enumerate}    

Clearly, the above vertices induce $G$.
\end{proof}

You can look at Figure~\ref{fig:simpleG} for a simple example of $G$ and Figure~\ref{fig:simpleindG} for the corresponding induced $G$. These were the graphs from Proposition~\ref{pr:simpleG}, but now you can see that they are special cases of Lemma~\ref{le:BipIndSubgraph}.\\

We are now ready to prove the main result of this paper.

\begin{theorem}
    Let $G$ be a bipartite graph. 
There exists $n,k\in\N$ such that for all 2-colorings of the edges of 
$B_{n,k}$, there exists an induced monochromatic $G$. 
\end{theorem}

\begin{proof} 
By Lemma~\ref{le:BipIndSubgraph}, there exists $a,b \in \N$ such that there exists an induced monochromatic $G$ in $B_{a, b},$ and by Theorem~\ref{th:IndBab}, there exists $n, k \in \N$ such that for all 2-colorings of the edges of $B_{n,k}$, there exists an induced monochromatic $B_{a,b}$. Now, let $\COL$ be a 2-coloring of the edges of $B_{n,k}$. Then, there is an induced monochromatic $B_{a,b}$, and since there is an induced $G$ in the $B_{a,b}$, there will be an induced monochromatic $G$ in $B_{n, k}$. 
\end{proof} 

\section{Acknowledgements} 

We presented this proof in a seminar. We thank
the people who listened to us and helped us 
polish this presentation. They are
Nik Carlson,
Ainesh Chatterjee, 
Andy Cui, 
Ilya Hajiaghayi, 
Cheng-Yuan Lee, 
Zhu Lipeng, 
John Purdy, and
Kelin Zhu.
We also thank proofreaders
Rob Brady and Soren Brown.

\end{document}